\documentclass[a4paper,12pt]{amsart}
\usepackage{amsfonts,stmaryrd,amssymb,amsmath,a4wide,verbatim,url,epsfig,color,hyperref,mathrsfs,graphicx,pgf,tikz,amsthm,mathtext,cite,enumerate,float}
\usepackage[british]{babel}
\usepackage[active]{srcltx}
\usepackage[T1,T2A]{fontenc}
\usepackage[utf8]{inputenc}
\usetikzlibrary{arrows}
\usepackage[matrix,arrow,curve]{xy}

\makeatletter
\@addtoreset{equation}{section}\makeatother

\newtheorem{theorem}{Theorem}[section]
\theoremstyle{definition}
\newtheorem{rem}{Remark}[section]
\theoremstyle{definition}
\newtheorem{lemma}{Lemma}[section]
\theoremstyle{definition}

\begin{document}

\title[Proof of a Conjecture of Wiegold]{Proof of a Conjecture of Wiegold}



\author[Alexander Skutin]{Alexander Skutin}

\maketitle

\section{Introduction}

In this short note we confirm a conjecture of James Wiegold \cite[4.69]{K}. The breadth $b(x)$ of an element $x$ of a finite $p$-group $G$ is defined by the equation $|G:C_G(x)| = p^{b(x)}$, where $C_G(x)$ is the centralizer of $x$ in $G$. We prove the following:

\begin{theorem}

Let $G$ be a finite $p$-group and let $|G'|>p^{n(n-1)/2}$ for some non-negative integer $n$. Then the group $G$ can be generated by the elements of breadth at least $n$.

\end{theorem}

An overview of this problem can be found in \cite{W}. Also M.R.Vaughan-Lee in \cite{Br} proved that if in a $p$-group $G$ we have $|G'|>p^{n(n-1)/2}$, then there exists an element in $G$ of breadth at least $n$.

In this article we prove that in the case $p\not= 2$ a more general result is true:

\begin{theorem}

Let $p\not= 2$ be a prime number. Let $G$ be a finite $p$-group and let $|G'|>p^{n(n-1)/2}$ for some non-negative integer $n$. Then the set of elements of breadth at least $n$ cannot be covered by two proper subgroups in $G$.

\end{theorem}

In the particular case $p = 2$ we also get a more general result:

\begin{theorem}

Let $G$ be a finite $2$-group and let $|G'|>2^{n(n-1)/2}$ for some non-negative integer $n$. Then the set of  elements of breadth at least $n$ cannot be covered by two proper subgroups in $G$, one of which has index at least $4$ in $G$.

\end{theorem}

So Theorem 1.1 is a consequence of the two Theorems 1.2 and 1.3.

\section{Proofs of Theorems 1.2 and 1.3}

The breadth $b_H(g)$ of an element $g$ of a finite $p$-group $G$ with respect to a subgroup $H\subseteq G$ is defined by the equation $|H:C_H(g)| = p^{b_H(g)}$, where $C_H(g) = \{h\in H| hg = gh\}$ is the centralizer of $g$ in $H$. By definition, $b(g) = b_G(g)$.

First we formulate Lemmas 2.1, 2.2, 2.3, which will be useful in the proofs of the Theorems 1.2 and 1.3.

\begin{lemma}

Let $G$ be a finite $p$-group and let $C$ be a finite subgroup of index $p$. Then for any element $g$ from the set $G\setminus C$ we get $\log_p |G'|\leq b(g) + \log_p|C'|$.

\end{lemma}
\newpage
\begin{proof}

The cardinality of the set $X = \{[g, c]| c\in C\}$ is not bigger than $p^{b(g)}$. So it is enough to prove that $G' = XC'$. This follows from the following properties of the set $XC'$ :

\begin{enumerate}
    \item $XC'$ is a subgroup in $G$ : $[g, c_1][g, c_2]C' = [g, c_1c_2]C'$;
    \item $XC'$ is a normal subgroup;
    \item Group $G/XC'$ is abelian.
\end{enumerate}
These facts imply Lemma 2.1.

\end{proof}

The next two lemmas are the well-known facts in theory of $p$-groups, so we state them without proof.

\begin{lemma}

Let $G$ be a finite $p$-group and let $|G:Z(G)|\leq p^2$. Then $\log_p|G'|\leq 1$.

\end{lemma}

\begin{lemma}

Let $G$ be a finite $p$-group and let $G = H_1\cup H_2\cup H_3$ for some three proper subgroups of $G$. Then $p = 2$ and $H_i$ has index $2$ in $G$ for $i=1, 2, 3$.

\end{lemma}

\subsection{Proof of Theorem 1.2}

Assume the converse. Let the proper subgroups $H_1$ and $H_2$ cover all the elements of breadth at least $n$. We will prove that then $|G'|\leq p^{n(n-1)/2}$. The proof is by induction on $|G|$. We can assume that $|G : H_i| = p$ (because every proper subgroup is contained in a subgroup of index $p$) and $H_1\not= H_2$ (because in every non-cyclic $p$-group, there are at least two different maximal proper subgroups). Consider any subgroup $C$ of $G$ such that $C$ has index $p$ in $G$ and $C\cap H_1 = C\cap H_2 = H_1\cap H_2$. Notice that the set $\{c\in C| b_C(c) = b(c)\}$ is contained in the subgroup $Y = \{c\in C| [c, g]\in C', \forall g\in G\}$. This is because if $b_C(c) = b(c)$, then for any $g\in G$ there exists $c'\in C$ such that $[c, g] = [c, c']$. Consider the case when $Y = C$. In this case the central subgroup $C/C'$ has a prime index in $G/C'$, so $G/C'$ is abelian and $G' = C'$. The rest follows from the induction hypothesis : the group $C$ has smaller order and all its elements of breadth at least $n$ are contained in the subgroup $C\cap H_1 = C\cap H_2$ (because $b_C(c)\leq b(c)$). Now consider the case $Y\not= C$. Notice that the set $\{c\in C| b_C(c)\geq n-1\}$ is contained in $(C\cap H_1)\cup Y$ (because $b_C(c) < b(c)$, when $c\notin Y$). Apply the induction hypothesis to the group $C$ and its proper subgroups $C\cap H_1$ and $Y$. We conclude that $\log_p|C'|\leq\frac{(n-2)(n-1)}{2}$. Consider any element $g$ not lying in $H_1\cup H_2\cup C$ (this is possible because of Lemma 2.3 and $p>2$). Its breadth is less than $n$ (because $g\notin H_1\cup H_2$), so from Lemma 2.1 we get that $\log_p|G'|\leq b(g) + \log_p|C'|\leq\frac{n(n-1)}{2}$. $\Box$

\begin{theorem}

Let $G$ be a finite $p$-group. Let given that for some integers $n\leq k + 1$

\begin{enumerate}

\item The set of all elements of the breadth at least $n$ can be covered by two proper subgroups of $G$.

\item The set of elements of the breadth at most $k$ generates $G$.

\end{enumerate}

Then $\log_p|G'|\leq\frac{(n-1)(n-2)}{2} + k$.

\end{theorem}

\begin{proof}

The proof is by induction on $|G|$. Let the subgroups $H_1$ and $H_2$ cover all the elements of the breadth at least $n$. We can assume that $|G : H_i| = p$ (because any proper subgroup is contained in the subgroup of index $p$) and $H_1\not= H_2$ (because in every non-cyclic $p$-group, there are at least two different maximal proper subgroups). Consider any subgroup $C$ of $G$ such that $C$ has index $p$ in $G$ and $C\cap H_1 = C\cap H_2 = H_1\cap H_2$. Notice that the set $\{c\in C| b_C(c) = b(c)\}$ is contained in the subgroup $Y = \{c\in C| [c, g]\in C', \forall g\in G\}$. It is because if $b_C(c) = b(c)$, then for any $g\in G$ there exists $c'\in C$ such that $[c, g] = [c, c']$. So we can conclude that the set $\{c\in C| b_C(c)\geq n-1\}$ is contained in $(C\cap H_1)\cup Y$ (because $b_C(c) < b(c)$ if $c\notin Y$). Consider the case $Y = C$. In this case the central subgroup $C/C'$ has a prime index in $G/C'$, so $G/C'$ is abelian and $G' = C'$. The rest follows from the induction hypothesis : group $C$ has smaller order and all its elements of the breadth at least $n$ are contained in the subgroup $C\cap H_1 = H_1\cap H_2$ (because $b_C(c)\leq b(c)$), also the set $C\setminus H_1$ is contained in the set $\{c\in C| b_C(c)\leq k\}$ (because $b_C(c)\leq b(c)\leq n-1\leq k$ for $c\in C\setminus H_1$) and so $C$ is generated by the elements of the breadth at most $k$ in subgroup $C$.

So we can assume that $Y$ is a proper subgroup of $C$. Consider the case $|C:Y| = p$ and consider a homomorphism $\pi : G\to G/C'$. It is clear that $\pi(Y)$ is the central subgroup of $\pi(G)$ and $|\pi(G):\pi(Y)|\leq p^2$. Apply Lemma 2.2 to $\pi(G)$, so we conclude that $\log_p|\pi(G)'|\leq 1$ and $\log_p|G'|\leq\log_p|C'|+1$. Apply the induction hypothesis to the group $C$ and to its two proper subgroups $H_1\cap H_2$, $Y$ so we get $\log_p|C'|\leq\frac{(n-3)(n-2)}{2} + n - 1$ (every element from the set $C\setminus H_1$ is of the breadth at most $n-1$ in $C$ and this set generates $C$, also for any element $g$ from the set $C\setminus (H_1\cup Y)$, we have $b_C(g)\leq n-2$). So if $k\geq 2$ we get $\log_p|G'|\leq\log_p|C'| + 1 \leq \frac{(n-3)(n-2)}{2}+n=\frac{(n-2)(n-1)}{2}+2\leq\frac{(n-2)(n-1)}{2}+k$ and the induction step is clear. In the case $k\leq 1$, we get $n=2$ (the case $n\leq 1$ is trivial) and the set $C\setminus (H_1\cup Y)$ is contained in the center of $C$ (because $b_C(g)\leq b(g)-1\leq 0$ for $g\in C\setminus (H_1\cup Y)$). It is clear that the set $C\setminus (H_1\cup Y)$ generates the central subgroup of index at most $p$ in $C$, so $C$ is abelian. From the conditions of Theorem 2.1 there exists an element $g\notin C$ such that $b(g)\leq k\leq 1$. So the group $C_G(g)$ has index at most $p$ in $G$. Also the subgroup $C_G(g)\cap C$ has index at most $p^2$ in $G$ and is central in $G$. Apply Lemma 2.2 to the group $G$, so we conclude that $\log_p|G'|\leq 1\leq\frac{(n-1)(n-2)}{2} + k$.

Eventually, we can assume that $|C : Y|\geq p^2$. So the set $C\setminus(H_1\cup Y)$ generates $C$, because otherwise if it generates proper subgroup $H$ of $C$, then $C = (H_1\cap C)\cup Y\cup H$ and from Lemma 2.3 we conclude that $|C:Y|\leq 2$, contradicting our assumption that $|C:Y|\geq p^2$. Also the set $C\setminus(H_1\cup Y)$ is contained in the set $\{g\in C| b_C(g)\leq n-2\}$, so $C$ is generated by the elements of the breadth at most $n-2$. And we can apply the induction hypothesis to $C$ and its two subgroups $H_1\cap H_2$, $Y$ and conclude that $\log_p|C'|\leq\frac{(n-3)(n-2)}{2}+n-2 = \frac{(n-2)(n-1)}{2}$. Also from the conditions of the Theorem 2.1 there exists an element $a\notin C$ such that $b(a)\leq k$. Apply Lemma 2.1 to $C$ and $a$, so $\log_p|G'|\leq b(a) + \log_p|C'|\leq\frac{(n-2)(n-1)}{2} + k$.

\end{proof}
\newpage
\subsection{Proof of Theorem 1.3}

Theorem 1.3 is a consequence of Theorem 2.1. Assume the converse. Let the proper subgroups $H_1$ and $H_2$ cover all the elements of breadth at least $n$ and $|G:H_2|\geq 4$. We will prove that then $|G'|\leq 2^{n(n-1)/2}$. The set $G\setminus(H_1\cup H_2)$ generates $G$, because otherwise if it generates proper subgroup $H$ of $G$, then $G = H_1\cap H_2\cup H$ and from Lemma 2.3 we conclude that $|G:H_2| \leq 2$, contradicting our assumption that $|G:H_2| \geq 4$. Now, notice that all the elements from $G\setminus(H_1\cup H_2)$ have breadth at most $n-1$ in $G$. So denote $k = n-1$ and apply Theorem 2.1 to $G$, $H_1$, $H_2$ and $n\le k+1$. We conclude that $\log_2|G'|\leq \frac{(n-1)(n-2)}{2}+k=\frac{n(n-1)}{2}$.

\begin{rem}

In fact, Theorem 1.2 is also a consequence of Theorem 2.1.

\end{rem}

The following Theorems 2.2, 2.3 are generalizations of Theorems 1.2 and 1.3, respectively. Their proofs are similar to the proofs of Theorems 1.2 and 1.3.

\begin{theorem}

Let $p$ be a prime number. Let $G$ be a finite $p$-group and let $|G'|>p^{n(n-1)/2}$ for some non-negative integer $n$. Then the set of elements of breadth at least $n$ cannot be covered by $p-1$ proper subgroups in $G$.

\end{theorem}

\begin{theorem}

Let $p$ be a prime number. Let $G$ be a finite $p$-group and let $|G'|>p^{n(n-1)/2}$ for some non-negative integer $n$. Then the set of elements of breadth at least $n$ cannot be covered by $p$ proper subgroups in $G$, one of which has index at least $p^2$ in $G$.

\end{theorem}

\section{Acknowledgements}

I am grateful to Anton A. Klyachko and to Michael Vaughan-Lee for useful discussions and the improvement of an earlier version of the manuscript.

\end{document}